\documentclass[11pt]{article}
\usepackage[T1]{fontenc}
\usepackage{lmodern,amsmath,amssymb,amsthm,mathtools}
\usepackage[backend=biber,style=numeric,sorting=nyt,sortcites=true]{biblatex}
\usepackage[margin=1in]{geometry}
\usepackage{microtype}
\usepackage[colorlinks=true,linkcolor=blue,urlcolor=blue,citecolor=blue]{hyperref}
\usepackage{enumitem}
\hypersetup{pdftitle={Single Lie Brackets on Closed Manifolds},
 pdfauthor={Derek Zeng},pdfsubject={Smooth vector fields and their Lie brackets}}

\newcommand{\Vect}{\mathfrak{X}}
\setlist{itemsep=3pt,topsep=5pt}
\newtheorem{theorem}{Theorem}[section]
\newtheorem{lemma}[theorem]{Lemma}
\newtheorem{proposition}[theorem]{Proposition}
\newtheorem{corollary}[theorem]{Corollary}
\theoremstyle{remark}
\newtheorem{remark}[theorem]{Remark}
\title{Single Lie Brackets on Closed Manifolds}
\author{Zeyu Zeng\thanks{derek6@illinois.edu} \\[3pt]
\small University of Illinois Urbana-Champaign}
\date{September 5, 2026}
\begin{document}
\maketitle
\begin{abstract}
We prove that every smooth vector field on a closed smooth manifold is a single Lie bracket. The proof combines a relative transport construction on the complement of a ball with an exact extension across transverse graph disks. For the local extension, we first arrange that an auxiliary bracket is transverse to the disks, straighten it by its flow, and then remove the resulting scalar discrepancy. The argument applies without an orientability assumption. It also implies that every smooth fiberwise linear function on the cotangent bundle is a Poisson bracket of two such functions.
\end{abstract}

\section{Introduction and main result}
\label{sec:introduction}

Let $M$ be a smooth manifold, and let $\mathfrak{X}(M)=\Gamma(TM)$ denote the real Lie algebra of smooth vector fields. We use the convention
$$
    [Y,Z](h)=Y(Zh)-Z(Yh),
    \qquad h\in C^\infty(M).
$$
The Lie algebra $\mathfrak{X}_c(M)$ of compactly supported smooth vector fields is perfect; see Janssens \cite[Corollary~1]{Janssens2016}. Thus every smooth vector field on a closed manifold is a finite sum of Lie brackets. We ask whether one bracket suffices.

\begin{theorem}
\label{thm:main}
Let $M$ be a closed smooth manifold. For every
$X\in\mathfrak{X}(M)$, there exist
$Y,Z\in\mathfrak{X}(M)$ such that
$$
    [Y,Z]=X.
$$
\end{theorem}

Here, closed means compact without boundary. No orientability assumption is required. For a Lie algebra $\mathfrak{g}$, its \emph{bracket width} is the least nonnegative integer $k$, if such an integer exists, for which every element of $[\mathfrak{g},\mathfrak{g}]$ is a sum of at most $k$ brackets. The theorem says that $\mathfrak{X}(M)$ has bracket width one when $M$ is nonempty and has positive dimension.

The question whether every smooth vector field is a single Lie bracket was posed by Greilhuber on Mathematics Stack Exchange \cite{Greilhuber2021}. This is an online discussion, not a peer-reviewed publication; we cite it as a source of the question and of the circle case recalled below. Theorem~\ref{thm:main} answers the question for closed manifolds.

The distinction between a sum of brackets and a single bracket also occurs in the algebraic category. Dubouloz, Kunyavskii, and Regeta \cite{DuboulozKunyavskiiRegeta2021} give examples of smooth affine varieties for which the Lie algebra of algebraic vector fields has bracket width greater than one. Our argument takes place in the real smooth category and uses smooth cutoff functions.

There is a related, but different, theory of smooth perfectness for diffeomorphism groups \cite{HallerRybickiTeichmann2013,HallerTeichmann2003}. Those results concern local factorizations of diffeomorphisms into products of group commutators, with smooth dependence on the diffeomorphism. Theorem~\ref{thm:main} concerns the surjectivity of the Lie bracket map. Our construction does not provide a continuous choice of $(Y, Z)$ as $X$ varies, or a factorization theorem for diffeomorphisms.

The difficulty is to make local solutions agree globally. In a coordinate box, writing the target field in coordinates and integrating its coefficients in one coordinate solves an equation of the form $[\partial_t,Z]=X$. Multiplying a solution by a cutoff, however, changes its bracket:
$$
    [\chi Y,\chi Z]
    =
    \chi^2[Y,Z]
    +\chi Y(\chi)Z
    -\chi Z(\chi)Y.
$$
We therefore use relative constructions that retain the prescribed pair on open neighborhoods of the sets along which solutions are joined.

The proof has two parts. First, a finite transport construction solves the bracket equation near a compact set on which an auxiliary vector field $S$ admits a function $f$ with $S(f)=1$. Its relative version retains prescribed data near a compact subset of $f^{-1}(0)$. A Morse preparation then gives a solution outside a small ball, equal to the pair $(\partial_t,t\partial_t)$ near an equator in a flow box of $X$.

Second, we extend this solution inward across two transverse graph disks. The local extension consists of three steps. We extend the pair while making the normal component of its bracket strictly positive in the chosen direction (Lemma~\ref{lem:positive}); use flow coordinates to make the bracket equal to the target field; and remove the scalar discrepancy in the outgoing pair by the constructions of Section~\ref{sec:scalar}. This gives the relative graph connector of Proposition~\ref{prop:graph}. It agrees with the prescribed pair at its entrance and with the reference pair at its exit and lateral boundary, so the two extensions can be joined to the reference solution in the remaining part of the ball.

Section~\ref{sec:transport} proves the transport statements. Section~\ref{sec:scalar} constructs the transverse and scalar corrections. Section~\ref{sec:graph} establishes the graph connector, and Sections~\ref{sec:equator}--\ref{sec:completion} complete the global proof. Section~\ref{sec:cotangent} records the consequence for fiberwise linear functions on the cotangent bundle. We use standard results on flows \cite[Chapter~9]{Lee2013}, Morse theory and tubular neighborhoods \cite{Hirsch1976}, and Sard's theorem \cite{Sard1942}.

All manifolds, functions, and vector fields are smooth. An \emph{exact pair} for a target $V$ on an open set is a pair $(Y, Z)$ satisfying $[Y, Z]=V$ there. Equality of \emph{full germs} along a compact set means equality on an open neighborhood of that set. A modification has \emph{compact transverse support} if its support throughout the finite cylinder is contained in one fixed compact set of transverse variables. Fields on a closed cylinder are understood to extend smoothly to an open neighborhood of it. We write $B_r$ for a Euclidean ball in the specified chart and $\operatorname{supp}$ for closed support.

A compact manifold has finitely many connected components, so it suffices to work on each component separately. In dimension zero the assertion is immediate. In dimension one each connected component is a circle. For completeness, the construction recorded in \cite{Greilhuber2021}, together with the comments there,
is as follows. Write
$$
    X=h(t)\partial_t
    \quad\text{on }\mathbb{R}/L\mathbb{Z},
    \qquad
    I=\int_0^L h(t)\,dt.
$$
If $I=0$, then
$$
    H(t)=\int_0^t h(s)\,ds
$$
is periodic and
$$
    [\partial_t,H\partial_t]=X.
$$
If $I\neq0$, set
$$
    r^2=\frac{|I|}{2\pi},
    \qquad
    \theta(t)=r^{-2}\int_0^t h(s)\,ds.
$$
Since
$$
    \theta(t+L)-\theta(t)
    =
    2\pi\operatorname{sgn}(I),
$$
the functions $r\cos\theta$ and $r\sin\theta$ are periodic,
and
$$
    [r\cos\theta\,\partial_t,\,
     r\sin\theta\,\partial_t]
    =
    r^2\theta'\partial_t
    =
    X.
$$
The remainder of the proof concerns $\dim M\geq2$.

\section{Finite transport with prescribed data on a transverse slice}\label{sec:transport}

Let $S$ be a complete vector field, with flow $\varphi_t$. For a vector field $W$, write
$$
 P_tW=(\varphi_t)_*W,\qquad
 A_\delta W=-\int_0^\delta P_tW\,dt.
$$
The flow differentiation formula (with the bracket convention fixed above; see~\cite[Chapter~9]{Lee2013}) gives
\begin{equation}
    \frac{d}{dt}P_tW=P_t[W,S]=[P_tW,S],
 \qquad [A_\delta W,S]=W-P_\delta W. \tag{2.1}\label{eq:transport}
\end{equation}
All integrals of fields here are coefficientwise integrals in a local trivialization; their coordinate independence follows from linearity.

\begin{lemma}[Finite transport]\label{lem:finite}
Suppose $M$ is compact, $K\subset O\subset M$, $K$ is compact, $O$ is open, and $S(f)=1$ on $O$ for a global smooth function $f$ and a global vector field $S$. For every global smooth vector field $E$ there is a global vector field $D$ with
$[D,S]=E$ on a neighborhood of $K$.

Moreover, if $\operatorname{supp}E\subset\{f\ge e\}$ for $e>0$, $D$ may be chosen to vanish on a neighborhood of $K\cap\{|f|<e/2\}$. The analogous assertion holds for support in $\{f\le-e\}$.
\end{lemma}

\begin{proof}
Choose $U$ with $K\subset U\Subset O$ and $\delta>0$ so small that $\varphi_{[-\delta,0]}(\overline U)\subset O$. Choose a smooth cutoff $\chi$ supported in $U$, equal to one on a neighborhood $U_*$ of $K$. Choose $W_0$ supported in $O$, equal to $E$ on $U$, by multiplying $E$ by another cutoff. If $W_0=0$, the result is immediate. Otherwise let $a_0=\min_{\operatorname{supp}W_0}f$ and recursively put
$$
 W_{j+1}=\chi P_\delta W_j.
$$
Inductively,
$$
 \operatorname{supp}W_j\subset O\cap\{f\ge a_0+j\delta\}.
$$
Indeed, if $y\in U$ and $P_\delta W_j(y)\ne0$, then $x=\varphi_{-\delta}(y)\in\operatorname{supp}W_j$, and the intervening short orbit stays in $O$. Hence $f(y)=f(x)+\delta\ge a_0+(j+1)\delta$. Taking closures proves the support assertion for $W_{j+1}$. Near $K$ we also have $W_{j+1}=P_\delta W_j$.

Choose an integer $q$ with $a_0+q\delta>\sup_U f$. The same argument gives $P_\delta W_{q-1}=0$ on $U$. Therefore
$$
 D=\sum_{j=0}^{q-1}A_\delta W_j
 \quad\Longrightarrow\quad
 [D,S]=\sum_{j=0}^{q-1}(W_j-P_\delta W_j)=E
$$
near $K$. This is a finite telescoping sum; it uses no long-time escape
assertion.

If $E$ is supported in $f\ge e$, choose $W_0$ with that same support condition. Every $W_j$ is then supported in $f\ge e$. For $p\in U$ with $|f(p)|<e/2$ and $0\le t\le\delta$, its short backward orbit stays in $O$ and
$$
 f(\varphi_{-t}(p))=f(p)-t<e.
$$
Thus $P_tW_j(p)=0$ for every $j$, and $D=0$ there. This proves the stated neighborhood conclusion. Apply this argument to $(-S,-f)$ and negate its primitive to obtain the negative-support version.
\end{proof}

\begin{lemma}[Relative transport]\label{lem:relative}
Under the hypotheses of Lemma~\ref{lem:finite}, let $A\subset K\cap f^{-1}(0)$ be compact. If $E$ vanishes on a neighborhood of $A$, there is a global $C$ such that
$$
 [C,S]=E\quad\text{near }K,\qquad C=0\quad\text{near }A.
$$
\end{lemma}

\begin{proof}
Let $U_0$ be an open neighborhood of $A$ where $E=0$, and choose $\eta_0\in C^\infty(M,[0,1])$ supported compactly in $U_0$, equal to one near $A$. Define the global smooth map and function
$$
 \Psi(p)=\varphi_{-f(p)}(p),\qquad \eta=\eta_0\circ\Psi.
$$
Since the flow preserves $S$, differentiation gives
$$
 d\Psi_p(S_p)=(1-S(f)(p))S_{\Psi(p)}.
$$
Consequently $S(\eta)=0$ on $O$. This calculation requires $S(f)=1$ only at $p$, not along the entire trajectory defining $\Psi$. Also $\Psi|_A$ is the identity, so $\eta=1$ on a neighborhood of $A$.

Choose $e>0$ so small that $\varphi_{[-e,e]}(\operatorname{supp}\eta_0)\subset U_0$. If $\eta(p)\ne0$ and $|f(p)|<e$, then
$$
 p=\varphi_{f(p)}(\Psi(p))\in U_0.
$$
It follows that $\eta E=0$ throughout the global strip $|f|<e$.
We may therefore split smoothly
$$
 \eta E=E_++E_-,\qquad
 \operatorname{supp}E_+\subset\{f\ge e\},\quad
 \operatorname{supp}E_-\subset\{f\le-e\},
$$
by taking its positive and negative $f$-parts.

Choose $Q$ with $[Q,S]=E$ near $K$ by Lemma~\ref{lem:finite}. Since
$S\eta=0$ on $O$,
$$
 [(1-\eta)Q,S]=(1-\eta)E\quad\text{near }K.
$$
Choose primitives $C_+,C_-$ for $E_+,E_-$ by the support-controlled parts of that lemma. They vanish near $A$, because $f=0$ there. Then
$$
 C=(1-\eta)Q+C_++C_-
$$
has both claimed properties.
\end{proof}

\section{Exact transverse shifts and scalar connectors}\label{sec:scalar}

In this section $T=\partial_t$. For a function $c$ of transverse variables,
write
$$
 R_c=(T,(t+c)T).
$$
More generally, for a transverse field $V$ independent of $t$,
the pair $(T,(t+c)T+V)$ has bracket $T$.

\begin{lemma}[Transverse shift]\label{lem:kick}
Let $I$ be a bounded interval and $K\in C_c^\infty(I)$. In any prescribed positive time interval one can construct an exact pair with bracket $T$ whose entrance and exit full germs are
$$
 (T,(t+c(u))T+R(u)\partial_u),\qquad
 (T,(t+c(u))T+(R(u)+K(u))\partial_u),
$$
respectively. The modification has compact transverse support. The same assertion holds with additional transverse variables as smooth parameters, provided $K$ has compact support in the full transverse coordinate box.
\end{lemma}

\begin{proof}
Put $H=-K$. Choose $B\in C_c^\infty(I)$ equal to one near $\operatorname{supp}H$, and $\alpha>0$ with $\alpha\|B'\|_\infty<1/4$. For an integer $N$ to be chosen, set
$$
 a=\frac{H}{N\alpha-H},\qquad
 P(u)=\int_{u_0}^u a(v)\,dv,\qquad D=BP,
$$
where $u_0$ is to the left of the support. Take $N$ sufficiently large that $N\alpha>2\|H\|_\infty$ and $\|D'\|_\infty<1/4$. The latter is possible: $a\to0$ uniformly, $P\to0$ uniformly on the fixed bounded interval, and $D'=B'P+Ba$. For $(A,C)\in[0,\alpha]\times[0,1]$,
$$
 F_{A,C}(u)=u+AB(u)+CD(u)
$$
is an increasing diffeomorphism, identity off a fixed compact set, because $\partial_uF_{A,C}>1/2$.

Starting at $(A,C)=(0,0)$, traverse this parameter rectangle counterclockwise $N$ times and end at $(0,0)$. Parametrize each edge smoothly with all derivatives zero at its endpoints; concatenate the edges and use constant endpoint collars. This is a smooth loop in any chosen positive time interval. Write $F(t,u)=F_{A(t),C(t)}(u)$. We claim
\begin{equation}
    \int \frac{F_t(t,u)}{F_u(t,u)}\,dt=H(u). \tag{3.1}\label{eq:loop}
\end{equation}
At a point with $H(u)\ne0$, we have $B=1$, $B'=0$, and $D'=a$. The two $C$-edges contribute opposite integrals, and the two $A$-edges contribute
$$
 \alpha-\frac{\alpha}{1+a}=\frac{\alpha a}{1+a}=\frac{H}{N}.
$$
At a point with $H(u)=0$, we have $a=0$ and $D'=B'P$ at that point. Writing $z=A+CP(u)$, the integrand is $B(u)z'/(1+B'(u)z)$, which is the derivative of a single-valued function of $z$ on the parameter range. Its closed-loop integral is zero.
This proves \eqref{eq:loop} everywhere.

Let $t_0$ be the entrance time, and put
$$
 \Phi(t,u)=(t,F(t,u)),\qquad
 Q(t,u)=\int_{t_0}^t\frac{F_t(s,u)}{F_u(s,u)}\,ds.
$$
Define the two fields to be the pushforwards by $\Phi$ of
$$
 \partial_t,\qquad
 (t+c(u))\partial_t+(R(u)-Q(t,u))\partial_u.
$$
The source bracket is
$$
 \partial_t-Q_t\partial_u
 =\partial_t-\frac{F_t}{F_u}\partial_u
 =\Phi^{-1}_*\partial_t.
$$
Thus the physical bracket is exactly $T$. On the entrance collar, $\Phi=\mathrm{id}$ and $Q=0$; on the exit collar, $\Phi=\mathrm{id}$ and $Q=H=-K$. This gives the required full germs.

For extra transverse variables $v$, replace $B,H,a,P,D$ by functions of $(u,v)$, integrate only in $u$, and use $\Phi(t,u,v)=(t,F(t,u,v),v)$. Choose $B$ compactly supported in a transverse box, equal to one near $\operatorname{supp}H$. The estimates are uniform on that compact box; for example,
$$
 \|D_u\|_\infty\le
 \bigl(|I|\|B_u\|_\infty+\|B\|_\infty\bigr)\|a\|_\infty
 \longrightarrow0.
$$ The triangular map is a diffeomorphism with a jointly smooth inverse because $F_u>1/2$. No smallness of derivatives in $v$ is needed. The same proof applies with $\partial_u$ throughout. An arbitrary incoming transverse field independent of $t$ is also allowed: its source coefficients have zero $t$-derivative, and it is recovered at both ends because $\Phi=\mathrm{id}$ there.
\end{proof}

\begin{lemma}[Scalar connector]\label{lem:scalar}
Let $c_{\mathrm{in}},c_{\mathrm{out}}$ be smooth functions on a transverse domain in $\mathbb R^m$, $m\ge1$, whose difference has compact support in its interior. In any prescribed positive time interval there is a smooth pair with bracket $T$, entrance germ $R_{c_{\mathrm{in}}}$ and exit germ $R_{c_{\mathrm{out}}}$. It retains that common germ off a compact subset of the transverse domain.
\end{lemma}

\begin{proof}
First suppose the difference $\Delta=c_{\mathrm{out}}-c_{\mathrm{in}}$ is supported in one transverse coordinate box, with one coordinate $u$. Choose a compactly supported function $g$ on that box satisfying $\partial_ug=\varepsilon$ near $\operatorname{supp}\Delta$, and put $K=\Delta/\varepsilon$. Such $g$ can be obtained by multiplying $\varepsilon u$ by a cutoff equal to one near that support. Its uniform norm can be arbitrarily small by taking $\varepsilon>0$ small.

The shear
$$
 S_g(s,u,v)=(s+g(u,v),u,v)
$$
preserves $T$. Conjugate a $+K\partial_u$ shift from Lemma~\ref{lem:kick} by this shear. For an incoming pair with second field $(t+c)T+R\partial_u$, the source second field is
$$
 (s+g+c-R\partial_ug)\partial_s+R\partial_u.
$$
After the shift and pushforward it is
$$
 (t+c+K\partial_ug)T+(R+K)\partial_u.
$$
An ordinary $-K\partial_u$ shift following it therefore changes only the scalar offset, by $K\partial_ug=\Delta$.

Here is a precise placement in a band of width $\ell$. Take $\|g\|_\infty<\ell/20$, perform the sheared shift using source times $[\ell/5,2\ell/5]$, and perform the ordinary inverse shift in $[3\ell/5,4\ell/5]$. The sheared stage lies between times $3\ell/20$ and $9\ell/20$. Fill the intervening regions with the explicit incoming or outgoing pairs $(T,(t+c)T+R\partial_u)$; all have bracket $T$. The endpoint collars guarantee smooth matching. The possibly large $K$ only increases the finite integer $N$ in Lemma~\ref{lem:kick}; it does not increase the support or the time width.

In general, split $\Delta$ by a finite partition of unity into functions supported in coordinate boxes compactly inside the transverse domain. Apply the complete scalar construction for each summand in successive short intervals. Each stage returns the transverse part to zero. Choose their total widths below the prescribed width. This also shows that any smaller collar whose closure lies in the open agreement region can be left fixed.
\end{proof}

\section{A full-boundary graph connector}\label{sec:graph}

Fix $L>0$ and $m\geq 1$, and write
$$
    C=[0,L]\times\mathbb R^m,
  \qquad (s,u)=(s,u^1,\ldots,u^m),
  \qquad S=\partial_s.
$$
Smooth fields on a closed cylinder are understood to extend smoothly to an open neighborhood. A transverse field is a field tangent to the slices $\{s=\mathrm{constant}\}$. Thus every pair admits a unique
decomposition
$$
  Y=aS+A,\qquad Z=bS+B,
  \qquad ds(A)=ds(B)=0.
$$
For $v=(a,b):C\to\mathbb R^2$, let
$$
  v_s=(\partial_s a,\partial_s b),
  \qquad
  d_uv=
  \begin{pmatrix}
    \partial_{u^1}a & \cdots & \partial_{u^m}a\\
    \partial_{u^1}b & \cdots & \partial_{u^m}b
  \end{pmatrix}.
$$
All norms and transverse gradients are Euclidean. With the convention $[Y,Z](h)=Y(Zh)-Z(Yh)$, we have
\begin{equation}\label{eq:normal}\tag{4.1}
  ds([Y,Z])
  =Y(b)-Z(a)
  =\det(v,v_s)+A(b)-B(a).
\end{equation}

\begin{lemma}[Positive normal component]
\label{lem:positive}
Fix $\varepsilon\in\{1,-1\}$. Let $E\subset C$ be closed and
suppose that, for some $R_0>0$,
$$
  \partial C\subset E,
  \qquad
  [0,L]\times\{u:|u|\geq R_0\}\subset E.
$$
Let $N$ be a relatively open neighborhood of $E$ in $C$.
Suppose that smooth vector fields $Y_*,Z_*$ are given on $N$ and
satisfy
$$
  [Y_*,Z_*]=\varepsilon S
  \qquad\text{on }N.
$$
Then there exist smooth vector fields $Y,Z$ on $C$,
agreeing with $Y_*,Z_*$ on a neighborhood of $E$, such that
$$
  \varepsilon\,ds([Y,Z])\geq\frac12
  \qquad\text{on }C.
$$
\end{lemma}

\begin{proof}
Put $\Omega=C\setminus E$, an open subset of
$(0,L)\times\mathbb R^m$. Every perturbation below will have compact
support in $\Omega$.

Choose a smooth function $\theta:C\to[0,1]$ such that
$\theta=1$ near $E$ and
$\operatorname{supp}_{C}\theta\subset N$.
Extend $\theta Y_*$ and $\theta Z_*$ by zero outside $N$, obtaining
a preliminary pair
$$
  Y_{\mathrm{pre}}=a^{\mathrm{pre}}S+A_0,
  \qquad
  Z_{\mathrm{pre}}=b^{\mathrm{pre}}S+B_0.
$$
Keep $A_0,B_0$ fixed until the last step. For a scalar pair $v=(a,b)$,
write
$$
  F(v)
  =\varepsilon\bigl(\det(v,v_s)+A_0(b)-B_0(a)\bigr),
  \qquad
  G(v)=|d_ua|^2+|d_ub|^2.
$$
Set
$$
  v^{\mathrm{pre}}=(a^{\mathrm{pre}},b^{\mathrm{pre}}),
  \qquad
  F_{\mathrm{pre}}=F(v^{\mathrm{pre}}).
$$
Since $F_{\mathrm{pre}}=1$ near $E$, the set
$$
  K_{\mathrm{pre}}=\{F_{\mathrm{pre}}\leq 3/4\}
$$
is a compact subset of $\Omega$. If it is empty, the preliminary
pair already suffices.

Our first objective is to arrange
\begin{equation}\label{eq:positivity-on-zero-set}\tag{4.2}
  F(v)>\frac12
  \qquad\text{on }\{d_uv=0\}.
\end{equation}
Once this holds, a transverse correction will establish the desired inequality everywhere.

Choose open sets
$$
  K_{\mathrm{pre}}\subset U\Subset V\Subset\Omega
$$
and $\rho\in C_c^\infty(V,[0,1])$ equal to $1$ near $\overline U$. For a constant matrix
$P\in\operatorname{Hom}(\mathbb R^m,\mathbb R^2)$, set
$$
  v_P=v^{\mathrm{pre}}+\rho Pu.
$$
On $U$, we have
$$
  d_uv_P=d_uv^{\mathrm{pre}}+P.
$$
Apply Sard's theorem to
$$
  d_uv^{\mathrm{pre}}:
  U\longrightarrow\operatorname{Hom}(\mathbb R^m,\mathbb R^2).
$$
We may choose $P$ arbitrarily small so that $-P$ is a regular value; equivalently, $d_uv_P$ is transverse to zero on $U$. The perturbation $\rho Pu$ tends to zero in $C^1$ on its fixed compact support as $P\to0$. We may therefore also require
$$
  \|F(v_P)-F_{\mathrm{pre}}\|_\infty<\frac18.
$$
Fix such a $P$, and write
$$
  v^{(1)}=v_P,
  \qquad
  F_1=F(v^{(1)}).
$$
Then $F_1>5/8$ outside $U$, so
$$
  K_1=\{F_1\leq1/2\}
$$
is a compact subset of $U$.

If $m\geq2$, the inequality $m+1<2m$ and transversality imply that $d_uv^{(1)}$ has no zeros in $U$. Since $F_1>1/2$ outside $U$, \eqref{eq:positivity-on-zero-set} holds, and we set $v^{(2)}=v^{(1)}$.

Suppose that $m=1$. The zeros of $v^{(1)}_u$ in $U$ are isolated. Consequently the compact set
$$
  \mathcal B=K_1\cap\{v^{(1)}_u=0\}
$$
is finite. Choose pairwise disjoint closed rectangles compactly contained in $U$, one about each point of $\mathcal B$, such that each rectangle contains no other zero of $v^{(1)}_u$.

Fix $p=(s_0,u_0)\in\mathcal B$ and its rectangle $Q=I_s\times I_u$, and abbreviate $v=v^{(1)}$ for the local construction. Choose
$$
  \chi\in C_c^\infty(\operatorname{int}I_u,[0,1]),
  \qquad
  \chi=1\quad\text{near }u_0.
$$
Since $I_s\times\operatorname{supp}\chi'$ contains no zero of $v_u$,
$$
  \lambda
  =\min_{I_s\times\operatorname{supp}\chi'}|v_u|
  >0.
$$
Fix $\delta>0$ with
$$
  \delta\|\chi'\|_\infty<\lambda/2.
$$
Any perturbation
$$
  \widetilde v=v+\eta(s)\chi(u),
  \qquad
  \eta\in C_c^\infty(\operatorname{int}I_s,\mathbb R^2),
  \qquad
  \|\eta\|_\infty<\delta,
$$
preserves the entire zero set of $v_u$. Indeed,
$$
  \widetilde v_u=v_u+\eta\chi'
$$
is unchanged where $\chi'=0$ and outside $Q$, whereas on $I_s\times\operatorname{supp}\chi'$,
$$
  |\widetilde v_u|
  \geq\lambda-\|\eta\|_\infty\|\chi'\|_\infty
  >\lambda/2.
$$

Let $J(x_1,x_2)=(-x_2,x_1)$. Choose $z\neq0$ with $|z-v(p)|<\delta/2$, and put
$$
  \alpha=z-v(p),
  \qquad
  \beta=\frac{2\varepsilon}{|z|^2}Jz-v_s(p).
$$
Take $\psi\in C_c^\infty((-1,1),[0,1])$ equal to $1$ near zero, and define
$$
  \eta_h(s)
  =
  \psi\!\left(\frac{s-s_0}{h}\right)
  \bigl(\alpha+(s-s_0)\beta\bigr).
$$
For sufficiently small $h>0$, its support lies in $\operatorname{int}I_s$ and
$$
  \|\eta_h\|_\infty
  \leq|\alpha|+h|\beta|
  <\delta,
  \qquad
  \eta_h(s_0)=\alpha,
  \qquad
  \eta_h'(s_0)=\beta.
$$
For $\widetilde v=v+\eta_h\chi$, it follows that
$$
  \widetilde v(p)=z,
  \qquad
  \widetilde v_s(p)=\frac{2\varepsilon}{|z|^2}Jz,
  \qquad
  \varepsilon\det\bigl(\widetilde v(p),\widetilde v_s(p)\bigr)=2.
$$
Moreover, $\widetilde v_u(p)=0$, so the transverse contribution to $F(\widetilde v)(p)$ vanishes. Hence
$$
  F(\widetilde v)(p)=2.
$$

Perform this modification in each chosen rectangle and denote the result by $v^{(2)}$. The zero set of the transverse derivative is preserved, and the coefficients near every zero outside $\mathcal B$ are unchanged. Thus \eqref{eq:positivity-on-zero-set} holds for $v^{(2)}$ in the case $m=1$ as well. No bound on the $s$-derivatives of these local perturbations is required.

Finally, write $v^{(2)}=(a,b)$ and set
$$
  F_2=F(v^{(2)}),
  \qquad
  G_2=G(v^{(2)}),
  \qquad
  K_2=\{F_2\leq1/2\}.
$$
The set $K_2$ is compactly contained in $\Omega$ and disjoint from $\{G_2=0\}$. If $K_2$ is empty, take
$$
  Y=aS+A_0,\qquad Z=bS+B_0.
$$
Otherwise choose $\chi_0\in C_c^\infty(\Omega,[0,1])$ equal to $1$ near $K_2$, and a finite constant
$$
  R\geq\max_{K_2}\frac{1/2-F_2}{G_2}.
$$
Define
$$
  A=A_0+\varepsilon R\chi_0\nabla_u b,
  \qquad
  B=B_0-\varepsilon R\chi_0\nabla_u a,
  \qquad
  Y=aS+A,\quad Z=bS+B.
$$
By \eqref{eq:normal},
$$
  \varepsilon\,ds([Y,Z])
  =F_2+R\chi_0G_2.
$$
This is at least $1/2$ on $K_2$ by the choice of $R$. Outside $K_2$, we have $F_2>1/2$, and the added term is nonnegative. Every perturbation has compact support in $\Omega$, so the prescribed full germs along $E$ are retained.
\end{proof}

To apply the lemma with prescribed endpoint and vertical collars, choose $E$ to contain smaller closed collars and a closed transverse exterior, all lying inside the regions where the compatible germs are given. Agreement on a neighborhood of $E$ retains these full germs. The construction works for every fixed $L>0$; no estimates uniform in $L$ are asserted or needed.

\begin{proposition}[Graph connector with every lateral side attached]
\label{prop:graph}
In coordinates $(t,u)$, let $X=\partial_t$ and $R=(X,tX)$. Let $P$ be a closed disk in the transverse coordinate domain, and let $\Gamma=\{t=g(u):u\in P\}$ be a smooth graph. Suppose an exact pair $(Y_0,Z_0)$ is given near $\Gamma$, with $[Y_0,Z_0]=X$, and equals $R$ on a neighborhood of its boundary. For either sign $\varepsilon$ and every sufficiently small $H>0$, there is an exact pair on a neighborhood of the graph cylinder
$$
 \{(t,u):u\in P,\quad t=g(u)+\varepsilon r,\quad0\le r\le H\}
$$
which equals the old pair near $r=0$, equals $R$ near $r=H$, and equals $R$ on a full vertical collar of $\partial P$.
\end{proposition}

\begin{proof}
Fix $\varepsilon\in\{1,-1\}$ and a sufficiently small $H>0$. Choose $L,\ell>0$ such that
$$
    2L+\ell<H.
$$
We shall first construct an exact pair up to a variable exit graph, correct its outgoing germ in a band of width $\ell$, and then extend by the reference pair.

Extend $g$ smoothly to $\mathbb R^m$, making it constant outside a compact set. Since the old pair agrees with $R$ near the boundary of the entrance graph, its entrance germ extends to all transverse variables by $R$. This extension is reference both outside $P$ and on a neighborhood of $\partial P$.

Consider the auxiliary cylinder $C_L=[0,L]\times\mathbb R^m$ with $(s,v)\in C_L$ and $S=\partial_s$ Under the coordinate map
$$
    J(s,v)=(g(v)+\varepsilon s,v),
$$
the target field and reference pair become
$$
    J^*X=\varepsilon S,
    \qquad
    J^*R=
    \bigl(\varepsilon S,\,[g(v)+\varepsilon s]\varepsilon S\bigr).
$$
Prescribe the pulled-back old germ near $s=0$, and prescribe $J^*R$ near $s=L$, throughout a full vertical collar of $\partial P$, and for all $v\notin P$. These prescriptions agree on their overlaps. After choosing disjoint endpoint collars, Lemma~\ref{lem:positive} gives a smooth pair $(\widehat Y,\widehat Z)$ retaining these germs
and satisfying
$$
    L_0=[\widehat Y,\widehat Z],
    \qquad
    ds(\varepsilon L_0)\ge\frac12.
$$
Set $V=\varepsilon L_0$. In every prescribed region, $V=S$.

We next use the flow of $V$ to straighten the bracket. Its coefficients are bounded on $C_L$: outside $P$ the field is $S$, while $[0,L]\times P$ is compact. The continuation theorem for smooth ODEs therefore rules out finite-time escape while a trajectory remains in the cylinder. Moreover, along every trajectory,
$$
    \frac{ds}{dr}=ds(V)\ge\frac12.
$$
Thus the trajectory starting at $(0,u)$ reaches $s=L$ exactly once, at a time $T(u)$ satisfying
$$
    0<T(u)\le 2L.
$$
The hit is transverse, so the implicit function theorem shows that $T$ is smooth.

Define
$$
    D_T=
    \{(r,u):u\in\mathbb R^m,\ 0\le r\le T(u)\},
    \qquad
    \Theta(r,u)=\operatorname{Fl}_V^{\,r}(0,u).
$$
This is a diffeomorphism from $D_T$ onto $C_L$. Indeed, the same continuation and monotonicity arguments show that every point of $C_L$ reaches $s=0$ uniquely under backward flow, in time at most $2L$. If $\tau(s,v)$ denotes this backward travel time and $\pi$
is the transverse projection, then
$$
    \Theta^{-1}(s,v)
    =
    \left(
        \tau(s,v),
        \pi\operatorname{Fl}_V^{-\tau(s,v)}(s,v)
    \right).
$$
The implicit function theorem at the bottom hit makes $\tau$, and hence this inverse, smooth. The endpoint germs allow these statements to be understood on neighborhoods of the closed cylinders.

By construction,
$$
    \Theta_*\partial_r=V=\varepsilon L_0,
    \qquad
    \Theta^*L_0=\varepsilon\partial_r.
$$
Now introduce the physical source coordinates through
$$
    F(r,u)=(g(u)+\varepsilon r,u).
$$
Since $F_*(\varepsilon\partial_r)=X$, the pair
$$
    (Y_1,Z_1)=F_*\Theta^*(\widehat Y,\widehat Z)
$$
satisfies
$$
    [Y_1,Z_1]
    =
    F_*\Theta^*L_0
    =
    X.
$$

We record the relative properties of this construction. Near the entrance, $V=S$, so $\Theta(r,u)=(r,u)$ for sufficiently small $r$. Consequently $(Y_1,Z_1)$ retains the old entrance germ. On the prescribed vertical collar, every integral curve of $V$ is vertical. By uniqueness, any trajectory meeting this collar must coincide with the vertical trajectory through that point; in particular, trajectories cannot cross $\partial P$. It follows that $\Theta$ restricts to a diffeomorphism
$$
    \{(r,u):u\in P,\ 0\le r\le T(u)\}
    \longrightarrow
    [0,L]\times P.
$$
On the collar, and also outside $P$, we have
$$
    \Theta(r,u)=(r,u),
    \qquad
    T(u)=L.
$$
Thus the physical pair is reference throughout the retained vertical collar.

It remains to identify and correct the exit germ. Define the transverse exit map $H_0$ by
$$
    \Theta(T(u),u)=(L,H_0(u)).
$$
Since $V=S$ near the auxiliary top, the flow has the explicit form
$$
    \Theta(r,u)=\bigl(L+r-T(u),H_0(u)\bigr)
$$
near $r=T(u)$. In particular, $\Theta^*S=\partial_r$ there.
Pulling back the auxiliary reference pair therefore gives
$$
    \Theta^*(\widehat Y,\widehat Z)
    =
    \left(
        \varepsilon\partial_r,\,
        \bigl[g(H_0(u))+\varepsilon(L+r-T(u))\bigr]
        \varepsilon\partial_r
    \right).
$$
Because $t=g(u)+\varepsilon r$, its physical expression is
$$
    (Y_1,Z_1)=\bigl(X,(t+c(u))X\bigr),
$$
where
\begin{equation}\label{eq:defect}\tag{4.3}
    c(u)=g(H_0(u))-g(u)+\varepsilon(L-T(u)).
\end{equation}
The function $c$ is smooth. On the retained collar and outside $P$, we have $H_0(u)=u$ and $T(u)=L$; hence $c=0$ there. In particular, $\operatorname{supp}c$ is a compact subset of $\operatorname{int}P$.

To apply Lemma~\ref{lem:scalar} above the variable exit graph, put
$$
    b(u)=g(u)+\varepsilon T(u),
    \qquad
    \rho=r-T(u),
$$
and use the coordinate map
$$
    K(\rho,u)=(b(u)+\varepsilon\rho,u).
$$
The graph $\rho=0$ is precisely the physical exit graph. In these coordinates,
$$
    K^*X=\varepsilon\partial_\rho,
$$
and the incoming pair is
$$
    K^*(Y_1,Z_1)
    =
    \left(
        \varepsilon\partial_\rho,\,
        [\rho+\varepsilon(b(u)+c(u))]\partial_\rho
    \right).
$$
Multiplying only the first field by $\varepsilon$ gives the normalized
pair
$$
    \left(
        \partial_\rho,\,
        [\rho+\varepsilon(b(u)+c(u))]\partial_\rho
    \right),
$$
whose bracket is $\partial_\rho$.

Apply Lemma~\ref{lem:scalar} in the band $0\le\rho\le\ell$, with offsets
$$
    c_{\mathrm{in}}=\varepsilon(b+c),
    \qquad
    c_{\mathrm{out}}=\varepsilon b.
$$
Their difference is compactly supported in $\operatorname{int}P$. We may therefore take the modification to have compact transverse support in $\operatorname{int}P$, retaining a smaller full collar of $\partial P$. Let $(A,B)$ be the resulting normalized pair. Thus
$$
    [A,B]=\partial_\rho,
$$
and its entrance and exit germs are the prescribed scalar pairs.

Undo the normalization and the coordinate change by setting
$$
    (Y_2,Z_2)=K_*(\varepsilon A,B).
$$
Then
$$
    [Y_2,Z_2]
    =
    K_*(\varepsilon\partial_\rho)
    =
    X.
$$
Its entrance germ agrees with $(Y_1,Z_1)$. At its exit,
$$
    (Y_2,Z_2)
    =
    K_*\left(
        \varepsilon\partial_\rho,\,
        [\rho+\varepsilon b(u)]\partial_\rho
    \right)
    =
    (X,tX)=R.
$$
On the retained vertical collar, the correction is also reference throughout.

The corrected upper graph is $r=T(u)+\ell$. Since
$$
    T(u)+\ell\le 2L+\ell<H,
$$
extend the pair by $R$ from this graph to $r=H$. Each gluing takes place where the adjoining pairs agree on an open
neighborhood, so the resulting pair is smooth and has bracket $X$. It retains the old germ near $r=0$, is reference near $r=H$, and is reference on a full vertical collar of $\partial P$. Finally, compactness of $P$ allows the endpoint neighborhoods to be chosen uniformly, giving the full collars asserted in the proposition.
\end{proof}

\section{Preparing a canonical equatorial belt}\label{sec:equator}

\begin{lemma}[Relative Morse preparation]\label{lem:morse} 
Suppose $M$ is closed and connected, $n\ge2$, and a coordinate neighborhood contains $\overline B_{2R}$ in $(t,u)$ coordinates. Let
$$
 A=\{t=0,\ |u|=R\}\cong S^{n-2}.
$$
There is a global Morse function $f$ equal to $t$ near $A$ such that all its critical points lie in $B_{R/4}$.
\end{lemma}

\begin{proof}
Extend the local function $t$ to a global smooth function using a cutoff, retaining it on a neighborhood of $\overline B_{2R}$. It has nonzero differential there. Make it Morse away from a smaller protected neighborhood as follows. Cover the compact region where critical points could occur by finitely many coordinate balls, with plateaus $\rho_j$ equal to one on smaller covering balls and supported away from the protected neighborhood. Add perturbations $\sum_j\rho_j a_j\cdot x_j$. At every possible critical point the parameter derivatives of the differential span the cotangent fiber. The universal critical locus is therefore smooth. Sard's theorem applied to its projection onto parameter space yields arbitrarily small parameters for which the differential is transverse to zero, equivalently all critical Hessians are nonsingular. Small parameters preserve the nonzero differential on the protected seam. Compactness gives finitely many critical points. Denote this Morse function by $f_0$.

Choose a small closed tubular neighborhood $\overline N$ of $A$, disjoint from $\overline B_{R/4}$ and contained where $f_0=t$. The tubular neighborhood can be taken to be a closed radius-$\rho$ normal disk bundle inside a radius-$2\rho$ tubular chart. Its complement is path connected. Here are details of this codimension-two fact. A path between two points outside $A$ can be perturbed off $A$: in finitely many product charts the normal space is two-dimensional, and a small generic perturbation of the one-dimensional path avoids its origin. This is also an immediate finite-parameter application of Sard with dimension count $1<2$. Thus $M\setminus A$ is path connected. In normal coordinates of radius up to $2\rho$, the continuous radial replacement
$$
 s\longmapsto\rho+s/2\quad(0<s\le2\rho),
 \qquad s\longmapsto s\quad(s\ge2\rho)
$$
pushes such paths off the closed radius-$\rho$ tube. Endpoints already outside that tube can be joined to their images by radial paths outside it. This proves path connectedness of $M\setminus\overline N$. The same path argument permits avoidance of finitely many additional points when $n\ge2$.

Move the critical points of $f_0$ one at a time to distinct points of $B_{R/4}$ by diffeomorphisms supported in $M\setminus\overline N$, avoiding the other marked points. Explicitly, join the current point to its destination by a path avoiding those points, subdivide the path into coordinate balls, and in each ball use the time-one flow of a constant coordinate vector multiplied by a cutoff equal to one along the short segment. The finitely many flows perform the move and fix a neighborhood of $\overline N$. Their composition $\Phi$ maps the original critical set into $B_{R/4}$. Then $f=f_0\circ\Phi^{-1}$ has precisely the required critical set and still equals $t$ near $A$.
\end{proof}

\begin{proposition}[Exterior solution canonical near the equator]
\label{prop:exterior}
Let $X=\partial_t$ on a coordinate neighborhood containing $\overline B_{2R}\subset M$, where $M$ is connected, closed, and $n\ge2$. There are global vector fields $Y_0,Z_0$ satisfying
$$
 [Y_0,Z_0]=X\quad\text{near }M\setminus\operatorname{int}B_{R/2},
$$
and having the full germ $(X,tX)$ near the equator $A=\{t=0,|u|=R\}$.
\end{proposition}

\begin{proof}
Take $f$ from Lemma~\ref{lem:morse}. Choose a Riemannian metric Euclidean near $A$. Since all critical points lie in $B_{R/4}$, the field
$$
 S=\frac{\operatorname{grad}f}{|\operatorname{grad}f|^2}
$$
is defined near $M\setminus B_{R/3}$. Extend it to a global smooth vector field, retaining that formula near $K=M\setminus\operatorname{int}B_{R/2}$. Thus $S(f)=1$ on an open neighborhood $O$ of $K$, and $S=X$ near $A$. Completeness follows from compactness.

Choose a global vector field $\overline D$ equal to $-tX$ near $A$. There
$[\overline D,S]=[-tX,X]=X$, so
$$
 E=X-[\overline D,S]
$$
vanishes near $A$. Also $A\subset K\cap f^{-1}(0)$. Lemma~\ref{lem:relative} gives $C$ with $[C,S]=E$ near $K$ and $C=0$ near $A$. Hence
$$
 D=\overline D+C,\qquad [D,S]=X\quad\text{near }K,
 \qquad D=-tX\quad\text{near }A.
$$
Set $Y_0=S$ and $Z_0=-D$. Their bracket is $X$, and near $A$ they are exactly $(X,tX)$, as claimed.
\end{proof}

\section{Completion of the proof by two caps}\label{sec:completion}

\begin{proof}[Proof of Theorem~\ref{thm:main} in dimension $n\ge2$]
If $X$ is identically zero on the connected component, take $Y=Z=0$. Otherwise choose a point where $X\ne0$. The flow-box theorem supplies coordinates $(t,u)$ with $X=\partial_t$, and we may choose $R>0$ with $\overline B_{2R}$ in this chart. Proposition~\ref{prop:exterior} supplies an exact exterior pair $(Y_0,Z_0)$ and an open reference neighborhood $V$ of the equator $A$, where that pair is $(X,tX)$.

For sufficiently small $\tau>0$, put
$$
 a=\sqrt{R^2-\tau^2}.
$$
The compact equatorial wedge
$$
 W_\tau=\{(t,u):t^2+|u|^2\le R^2,\ |u|\ge a\}
$$
lies in $V$. To see this, these wedges converge to $A$ as $\tau\to0$: their points satisfy $|t|\le\tau$ and $a\le|u|\le R$. Choose $\tau$ strictly small enough that the wedge has an open neighborhood in $V$. The two complementary boundary caps are the smooth graphs
$$
 t=q(u),\qquad t=-q(u),\qquad
 q(u)=\sqrt{R^2-|u|^2},\qquad |u|\le a.
$$
They are genuine transverse graph disks, since $q\ge\tau>0$; no graph is taken through the equator itself. The old pair is given on a full neighborhood of both caps, and equals the reference on full collars of their entire rims.

Apply Proposition~\ref{prop:graph} to the north cap with inward sign $\varepsilon=-1$, and to the south cap with inward sign $\varepsilon=+1$, using one sufficiently small uniform height $0<H<\tau/2$. These give exact pairs in the two shells
\begin{align*}
    N&=\{|u|\le a,\ q(u)-H\le t\le q(u)\},\\
    S&=\{|u|\le a,\ -q(u)\le t\le-q(u)+H\}.
\end{align*}
Both shells lie in $\overline B_R$ and are disjoint: the first has $t\ge\tau-H>0$, while the second has $t\le-\tau+H<0$. Each pair equals the old pair on its outer full collar and equals the reference on its inner full collar and on a full collar of $|u|=a$.

Define the pair to be the reference $(X,tX)$ on the middle region
$$
 C=\{|u|\le a,\ -q(u)+H\le t\le q(u)-H\}
$$
and on $W_\tau$. Use the old pair on $M\setminus\operatorname{int}B_R$. These regions cover $M$. At the outer cap boundaries the pairs agree with the old full germs. At the inner cap boundaries they agree with the reference full germs. At all lateral boundaries and their corners, both pairs equal the reference on open neighborhoods. On the remaining spherical belt the old pair is already reference because
$W_\tau\subset V$.

Thus the specified pieces define globally smooth fields $Y,Z$. Their bracket equals $X$ on the exterior, on the two shells, and on the reference regions, where $[X,tX]=X$. Hence $[Y,Z]=X$ everywhere.
\end{proof}

\section{A cotangent-bundle consequence}\label{sec:cotangent}

The theorem has a direct interpretation on the cotangent bundle. For $V\in\Vect(M)$ define the smooth fiberwise linear function
$$
 \ell_V:T^*M\longrightarrow\mathbb R,\qquad
 \ell_V(\alpha_x)=\alpha_x(V_x).
$$
Every smooth fiberwise linear function has this form for a unique $V$: in local cotangent coordinates its coefficients are the smooth functions $V^j(q)=\partial_{p_j}F(q,0)$. In canonical cotangent coordinates $(q^i,p_i)$, fix the Poisson convention
$$
 \{F,G\}=\sum_i\left(
 \frac{\partial F}{\partial q^i}\frac{\partial G}{\partial p_i}
 -\frac{\partial F}{\partial p_i}\frac{\partial G}{\partial q^i}
 \right).
$$
This is the canonical cotangent Poisson structure with $\{q^i,p_j\}=\delta^i_j$; background on the canonical symplectic structure may be found in~\cite[Chapter~22]{Lee2013}. Since $\ell_V(q,p)=\sum_jp_jV^j(q)$, direct differentiation gives
$$
 \{\ell_U,\ell_V\}
 =\sum_{i,j}p_j\bigl(V^i\partial_iU^j-U^i\partial_iV^j\bigr)
 =-\ell_{[U,V]}.
$$
The sign is fixed by our explicit Poisson convention.

\begin{corollary}\label{cor:cotangent} 
Let $M$ be closed. Every smooth fiberwise linear function $F$ on $T^*M$ can be written as $F=\{G,H\}$, where $G$ and $H$ are smooth fiberwise
linear functions on $T^*M$.
\end{corollary}

\begin{proof}
Write $F=\ell_X$ and choose $Y,Z$ with $[Y,Z]=X$ by Theorem~\ref{thm:main}. Then $F=\ell_{[Y,Z]}=\{\ell_Z,\ell_Y\}$, so take $G=\ell_Z$ and $H=\ell_Y$.
\end{proof}

\begin{remark}
The representation theorem is existential in the smooth category. Each transport sum, parameter loop construction, and set of local jet corrections is finite. The widths used in the graph connectors are fixed and positive before the final fields are constructed. Derivatives of the factors may grow as a connector is made thinner; no limiting construction or uniform derivative estimate is needed. In particular, the argument does not supply a continuous choice of bracket factors as $X$ varies.
\end{remark}

% -------------------------------------------------
% Bibliography
% -------------------------------------------------
\printbibliography

\end{document}